\documentclass{amsart}
\usepackage{amssymb, amsmath, amsthm, mathrsfs}
\usepackage{cite, hyperref}
\makeatletter
\renewcommand\@makefntext[1]{\noindent #1}
\makeatother

\usepackage[headsep=0.25truein, top=1.25truein, bottom=1truein, hmargin={1truein,1truein}]{geometry}
\allowdisplaybreaks[4]
\numberwithin{equation}{section}

\usepackage{graphicx,xcolor}
\definecolor{db}{RGB}{23,20,219}
\definecolor{dg}{RGB}{2,101,15}
\colorlet{sectitlecolor}{red!60!black}
\colorlet{sectboxcolor}{cyan!30}
\colorlet{secnumcolor}{orange}
\usepackage[explicit]{titlesec}
\titleformat{\section}
{\sffamily\color{sectitlecolor}\Large\bfseries\filcenter}{}{2em}{\thesection.\quad #1}%
\newtheoremstyle{mytheorem}{5pt}{}{\color{db}}{}{\color{db}\bfseries}{}{ }{}
\theoremstyle{mytheorem}
\newtheorem*{maintheorem}{Main Theorem}
\newtheorem{theorem}{Theorem}[section]
\newtheorem{corollary}[theorem]{Corollary}
\newtheorem{lemma}[theorem]{Lemma}
\newtheorem{proposition}[theorem]{Proposition}
\theoremstyle{definition}

\theoremstyle{example}

\theoremstyle{remark}
\newtheorem{remark}[theorem]{Remark}
\numberwithin{equation}{section}

\usepackage[displaymath, mathlines]{lineno} 

\newcommand{\ov}[1]{\overline{#1}}

\newcommand{\bm}[1]{\boldsymbol{#1}}
\def\Li{{\textrm{Li}}}

\begin{document}
\footnotetext{%
Date: 2022-03-26; Version 7. 
}

\title{Weighted Sum Formulas from Shuffle Products of Multiple Zeta-star Values}

\author{Kwang-Wu Chen}
\address{Department of Mathematics, University of Taipei, 10048 Taipei, Taiwan}
\email{kwchen@utaipei.edu.tw}
\thanks{The first author (corresponding author)
was funded by the Ministry of Science and Technology,
Taiwan, R.O.C., under Grant MOST 110-2115-M-845-001.}

\author{Minking Eie}
\address{Department of Mathematics, National Chung Cheng University, 168 University Road, Min-Hsiung, Chia-Yi 62145, Taiwan}
\email{minkingeie@gmail.com}


\begin{abstract}
In this paper, we are going to perform the shuffle products of
$Z_-(n) = \sum_{a+b=m} (-1)^{b} \zeta(\{1\}^{a},b+2)$ and 
$Z_+^\star(n) = \sum_{c+d=n} \zeta^{\star}(\{1\}^{c},d+2)$
with $m+n = p$. The resulted shuffle relation is a weighted sum formula given by
\begin{align*}
   &\frac{(p+1)(p+2)}{2} \zeta(p+4)
   =\sum_{m+n=p} \sum_{|\boldsymbol{\alpha}|=p+3}
    \zeta(\alpha_{0}, \alpha_{1}, \ldots, \alpha_{m}, \alpha_{m+1}+1) \\
  &\qquad \times \sum_{a+b+c=m} 
  \Bigl( W_{\bm\alpha}(a,b,c) + W_{\bm\alpha}(a,b,c=0) 
  + W_{\bm\alpha}(a=0,b,c) + W_{\bm\alpha}(a=0,b=m,c=0) \Bigr),
\end{align*}
where $W_{\bm\alpha}(a,b,c) = 2^{\sigma(a+b+1)-\sigma(a)-(b+1)} (1-2^{1-\alpha_{a+b+1}})$,
with $\sigma(r) = \sum_{j=0}^{r} \alpha_{j}$.
\end{abstract}

\keywords{Multiple zeta value, Shuffle product}

\subjclass[2020]{Primary: 11M32; Secondary: 05A15, 33B15.}

\maketitle

\section{Introduction}\label{sec1}
For an $r$-tuple $\boldsymbol{\alpha} = (\alpha_{1}, \alpha_{2}, \ldots, \alpha_{r})$ 
of positive integers with $\alpha_{r} \geq 2$, a multiple zeta value $\zeta(\bm\alpha)$
and a multiple zeta-star value $\zeta^\star(\bm\alpha)$
are defined to be \cite{E09, E13}
\[
  \zeta(\boldsymbol{\alpha})
  = \sum_{1 \leq k_{1} < k_{2} < \cdots < k_{r}} k_{1}^{-\alpha_{1}}
    k_{2}^{-\alpha_{2}} \cdots k_{r}^{-\alpha_{r}},
\]
and
\[
  \zeta^{\star}(\boldsymbol{\alpha})
  = \sum_{1 \leq k_{1} \leq k_{2} \leq \cdots \leq k_{r}} k_{1}^{-\alpha_{1}}
    k_{2}^{-\alpha_{2}} \cdots k_{r}^{-\alpha_{r}}.
\]
We denote the parameters $w(\bm\alpha)=\mid\bm\alpha\mid 
= \alpha_{1} + \alpha_{2} + \cdots + \alpha_{r}$,
$d(\bm\alpha)=r$, and $h(\bm\alpha)=\#\{i\mid \alpha_i>1, 1\leq i\leq r\}$,
called respectively the weight, the depth, and the height of $\bm\alpha$
(or of $\zeta(\bm\alpha)$, or of $\zeta^\star(\bm\alpha)$).
For any multiple zeta value $\zeta(\bm\alpha)$, we put a bar on top of $\alpha_j$
($j=1,2,\ldots,r$) if there is a sign $(-1)^{k_j}$ appearing in the numerator of 
the summand \cite{Teo2018, Xu2019}. We call it an alternating multiple zeta value.
For example, an alternating multiple zeta value 
\[
\zeta(\ov{1},\{1\}^{m},\ov{2})=\sum_{1\leq k_1<k_2<\cdots<k_{m+2}}
\frac{(-1)^{k_1+k_{m+2}}}{k_1\cdots k_{m+1}k_{m+2}^2},
\]
where $\{a\}^{k}$ is $k$ repetitions of $a$. 
Its weight is $m+3$, the height is one, and the depth is $m+2$.

Due to Kontsevich, multiple zeta values can be represented by iterated integrals over a simplex of dimension weight:
\[
  \zeta(\alpha_{1}, \alpha_{2}, \ldots, \alpha_{r})
  = \int_{E_{|\boldsymbol{\alpha}|}} \Omega_{1} \Omega_{2} \cdots
    \Omega_{|\boldsymbol{\alpha}|} \quad \textrm{or} \quad
    \int_{0}^{1} \Omega_{1} \Omega_{2} \cdots \Omega_{|\boldsymbol{\alpha}|}
\]
with $E_{|\boldsymbol{\alpha}|}: 0 < t_{1} < t_{2} < \cdots
< t_{|\boldsymbol{\alpha}|} < 1$ and
\[
  \Omega_{j} = \begin{cases}
  \dfrac{dt_{j}}{1-t_{j}}
    &\textrm{if $j = 1, \alpha_{1}+1, \alpha_{1}+\alpha_{2}+1, \ldots, \alpha_{1}+\alpha_{2}+\cdots+\alpha_{r-1}+1$}, \\
  \dfrac{dt_{j}}{t_{j}} &\textrm{otherwise}.
  \end{cases}
\]
Sometime, we express them as words with just two alphabets:
\[
  a\{b\}^{\alpha_{1}-1} a\{b\}^{\alpha_{2}-1} \cdots a\{b\}^{\alpha_{r}-1}.
\]
The dual then is obtained by exchanging $a$, $b$ and reversing the order. For example, we express $\zeta(5)$ as $ab^4$, its dual is $a^4b$ which is $\zeta(1,1,1,2)$.

Once multiple zeta values are expressed as iterated integrals, the shuffle product of two multiple zeta values can be defined as
\[
  \int_{0}^{1} \Omega_{1} \Omega_{2} \cdots \Omega_{m}
    \int_{0}^{1} \Omega_{m+1} \Omega_{m+2} \cdots \Omega_{m+n}
  = \sum_{\sigma} \int_{0}^{1} \Omega_{\sigma(1)} \Omega_{\sigma(2)} \cdots
    \Omega_{\sigma(m+n)},
\]
where $\sigma$ ranges over all permutations on the set $\{ 1,2,\ldots,m+n \}$ that preserve the orderings of $\Omega_{1} \Omega_{2} \cdots \Omega_{m}$ and $\Omega_{m+1} \Omega_{m+2} \cdots \Omega_{m+n}$. This is equivalent to saying that for all $1 \leq i < j \leq m$ or $m+1 \leq i < j \leq m+n$, we have $\sigma^{-1}(i) < \sigma^{-1}(j)$. It is clear that the shuffle product of two multiple zeta values of weight $m$ and $n$, respectively, will produce $\binom{m+n}{n}$ multiple zeta values of weight $m+n$.

In this paper, we will investigate the following integral
\[
   \frac{1}{p!} \int_{0<t_1<t_2<1\atop 0<u_1<u_2<1}
    \left( \log \frac{1-t_{1}}{1-u_{2}} + \log \frac{t_{2}}{t_{1}}
    + \log \frac{u_{2}}{u_{1}} \right)^{p} \frac{dt_{1} dt_{2}}{(1-t_{1})t_{2}} \frac{du_{1} du_{2}}{(1-u_{1})u_{2}}.
\]
We derive our main theorem by finding its different representations.

\begin{maintheorem}
For integer $p \geq 0$, we have

\begin{align*}
   &\frac{(p+1)(p+2)}{2} \zeta(p+4)
   =\sum_{m+n=p} \sum_{|\boldsymbol{\alpha}|=p+3}
    \zeta(\alpha_{0}, \alpha_{1}, \ldots, \alpha_{m}, \alpha_{m+1}+1) \\
  &\qquad \times \sum_{a+b+c=m} 
  \Bigl( W_{\bm\alpha}(a,b,c) + W_{\bm\alpha}(a,b,c=0) 
  + W_{\bm\alpha}(a=0,b,c) + W_{\bm\alpha}(a=0,b=m,c=0) \Bigr),
\end{align*}

\noindent where $W_{\bm\alpha}(a,b,c) = 2^{\sigma(a+b+1)
-\sigma(a)-(b+1)} (1-2^{1-\alpha_{a+b+1}})$,
with $\sigma(r) = \sum_{j=0}^{r} \alpha_{j}$.
\end{maintheorem}

We note that the second, third, and fourth weight in the sum formula are 

\begin{align*}
\sum_{a+b+c=m}W_{\bm\alpha}(a,b,c=0)&=\sum_{a+b=m}W_{\bm\alpha}(a,b,0)
=\sum^m_{a=0}
2^{\sigma(m+1)-\sigma(a)-(m-a+1)} (1-2^{1-\alpha_{m+1}}),\\
\sum_{a+b+c=m}W_{\bm\alpha}(a=0,b,c)&=\sum_{b+c=m}W_{\bm\alpha}(0,b,c)
=\sum^m_{b=0}
2^{\sigma(b+1)-\alpha_{0}-(b+1)} (1-2^{1-\alpha_{b+1}}),\\
\sum_{a+b+c=m}W_{\bm\alpha}(a=0,b=m,c=0)&=W_{\bm\alpha}(0,m,0).
\end{align*}

For brevity, we hereafter use the lowercase English letters and lowercase
Greek letters in the summation, with or without subscripts, to denote 
the nonnegative and positive integers, unless otherwise specified.

In fact, our main theorem is from the integral
\[
  J_{p}
  = \frac{1}{p!} \int_{0<t_1<t_2<1\atop 0<u_1<u_2<1}
    \left( \log \frac{1-t_{1}}{1-u_{2}} + \log \frac{t_{2}}{t_{1}}
    + \log \frac{u_{2}}{u_{1}} \right)^{p} \frac{dt_{1} dt_{2}}{(1-t_{1})t_{2}} \frac{du_{1} du_{2}}{(1-u_{1})u_{2}}.
\]

As a replacement of shuffle product, we decompose the domain 
$\{(t_1,t_2,u_1,u_2)\in\mathbb R^4\mid 0<t_1<t_2<1,0<u_1<u_2<1\}$ 
into 6 disjoint simplices of dimension 4:

\begin{gather*}
  D_{1}: 0 < t_{1} < t_{2} < u_{1} < u_{2} < 1, \quad
  D_{2}: 0 < u_{1} < u_{2} < t_{1} < t_{2} < 1, \\
  D_{3}: 0 < t_{1} < u_{1} < t_{2} < u_{2} < 1, \quad
  D_{4}: 0 < t_{1} < u_{1} < u_{2} < t_{2} < 1, \\
  D_{5}: 0 < u_{1} < t_{1} < u_{2} < t_{2} < 1, \quad
  D_{6}: 0 < u_{1} < t_{1} < t_{2} < u_{2} < 1.
\end{gather*}

\noindent So
\[
  J_{p} = \sum_{j=1}^{6} \mathbb J(j)
\]
with
\[
  \mathbb J(j)
  = \frac{1}{p!} \int_{D_{j}} \left( \log \frac{1-t_{1}}{1-u_{2}}
    + \log \frac{t_{2}}{t_{1}} + \log \frac{u_{2}}{u_{1}} \right)^{p} \frac{dt_{1} dt_{2}}{(1-t_{1})t_{2}} \frac{du_{1} du_{2}}{(1-u_{1})u_{2}}.
\]
The main task now is to evaluate $\mathbb J(j)$ one by one in terms of multiple zeta values. 
The weighted sum in our main theorem is $\sum_{j=3}^6\mathbb J(j)$.
The crucial step is to prove
\begin{equation}\label{eq.j12}
  \mathbb J(1) + \mathbb J(2) = J_{p} - \frac{(p+1)(p+2)}{2} \zeta(p+4)
\end{equation}
and the resulted shuffle relation is
\begin{equation}\label{eq.j36}
  \sum_{j=3}^{6} \mathbb J(j) = \frac{(p+1)(p+2)}{2} \zeta(p+4).
\end{equation}

We organize this paper as follows: We give some preliminaries and auxiliary tools 
in Section \ref{sec.2}. In Section \ref{sec.3}, 
we break $J_p$ into six parts and evaluate $\mathbb J(1)$
and $\mathbb J(2)$ in terms of multiple zeta values. In the next section we express
$\mathbb J(j)$ for $3\leq j\leq 6$ as the part of the weighted sum of our main theorem.
We prove the crucial step Eq.\,(\ref{eq.j12})
in Section \ref{sec.5}. In Section \ref{sec.6}, we give another application of the identity
Eq.\,(\ref{eq.j36}) to another sum formula wihich
is also proved by a duality theorem due to Ohno \cite{Ohno1999}. 
In Section \ref{sec.7}, we relate $\mathbb J(3)$
with the weighted sum 
\[
\sum_{|\bm\alpha|=n+3}2^{\alpha_2}\zeta^\star(\alpha_1,\{1\}^m,\alpha_2+1).
\]
In the final section, we will give some formulas related to 
weighted alternating Euler sums, e.g., 
\[
\sum_{m+n=p}n\zeta(m+1,\ov{n+1}),\quad\mbox{and}\quad
\sum_{m+n=p}n(n-1)\zeta(m+1,\ov{n+1}).
\]

\section{Some Preliminaries and Auxiliary Tools}\label{sec.2}
Here we list a general integral representation which we use frequently.

\begin{proposition} \textup{\cite{EW2008}} 
For nonnegative integers $b_{1}, b_{2}, \ldots, b_{r}, b_{r+1}$, we have
\[
  \zeta(b_{1}+1, b_{2}+1, \ldots, b_{r}+1, b_{r+1}+2)
  = \frac{1}{b_{1}! b_{2}! \cdots b_{r+1}!} \int_{E_{r+2}} \prod_{j=1}^{r+1}
    \frac{\mathrm{d}t_{j}}{1-t_{j}} \left( \log \frac{t_{j+1}}{t_{j}} \right)^{b_{j}} \frac{\mathrm{d}t_{r+2}}{t_{r+2}},
\]
where $E_{r+2}$ is a simplex defined as $0 < t_{1} < t_{2} < \cdots < t_{r+2}$.
\end{proposition}

\begin{proposition}\cite[Proposition 2.1]{EW2008}\label{prop.int}
For integers $p,q,r,\ell \geq 0$,

\begin{align*}
  &\quad \sum_{|\boldsymbol{\alpha}|=q+r+1}
    \zeta(\{1\}^{p}, \alpha_{0}, \alpha_{1}, \ldots, \alpha_{q}+\ell+1) \\
  &= \frac{1}{p!q!r!\ell!} \int_{0 < t_{1} < t_{2} < 1}
    \left( \log \frac{1}{1-t_{1}} \right)^{p}
    \left( \log \frac{1-t_{1}}{1-t_{2}} \right)^{q}
    \left( \log \frac{t_{2}}{t_{1}} \right)^{r}
    \left( \log \frac{1}{t_{2}} \right)^{\ell}
    \frac{dt_{1} dt_{2}}{(1-t_{1})t_{2}}.
\end{align*}
\end{proposition}

The well-known sum formula due to Granville \cite{G96} asserted that
\[
  \sum_{|\boldsymbol{\alpha}|=q+r+1}
    \zeta(\alpha_{0}, \alpha_{1}, \ldots, \alpha_{q}+1)
  = \zeta(q+r+2).
\]
If we let $p = \ell = 0$ in the formula of Proposition \ref{prop.int}, then we have
\begin{equation}
  \sum_{|\boldsymbol{\alpha}|=q+r+1}
    \zeta(\alpha_{0}, \alpha_{1}, \ldots, \alpha_{q}+1)
  = \frac{1}{q!r!} \int_{0 < t_{1} < t_{2} < 1}
    \left( \log \frac{1-t_{1}}{1-t_{2}} \right)^{q}
    \left( \log \frac{t_{2}}{t_{1}} \right)^{r}
    \frac{dt_{1} dt_{2}}{(1-t_{1})t_{2}}.
\end{equation}
Therefore the integral 
\[
  \frac{1}{q!r!} \int_{0 < t_{1} < t_{2} < 1}
    \left( \log \frac{1-t_{1}}{1-t_{2}} \right)^{q}
    \left( \log \frac{t_{2}}{t_{1}} \right)^{r}
    \frac{dt_{1} dt_{2}}{(1-t_{1})t_{2}}=\zeta(q+r+2).
\]

Our main results are obtained by the convolution of the following sums:
\begin{equation}\label{eq.zz}
  Z_-(m) = \sum_{a+b=m} (-1)^{b} \zeta(\{1\}^{a},b+2) \quad \textrm{and} \quad
  Z^\star_+(n) = \sum_{c+d=n} \zeta^{\star}(\{1\}^{c},d+2).
\end{equation}
A special case of results proved by Le and Murakami \cite{LM95} is stated as follows.
\begin{equation}\label{eq.a.1}
\sum_{a+b=2w-2}(-1)^{a+1}\zeta(\{1\}^a,b+2)
=2\zeta(\ov{2w}).
\end{equation}

Both $Z_{-}(n)$, and $Z^\star_{+}(n)$ 
are sums of multiple zeta(-star) values of height one and can be expressed as double integrals
(see \cite{CE2021})

\begin{align*}
  Z_-(n)
  &= \frac{1}{n!} \int_{E_{2}}
    \left( \log \frac{1}{1-t_{1}} - \log \frac{t_{2}}{t_{1}} \right)^{n} \frac{\mathrm{d}t_{1} \mathrm{d}t_{2}}{(1-t_{1})t_{2}}, \\
  Z^\star_+(n)
  &= \frac{1}{n!} \int_{E_{2}}
    \left( \log \frac{1}{1-t_{2}} + \log \frac{t_{2}}{t_{1}} \right)^{n} 
    \frac{\mathrm{d}t_{1} \mathrm{d}t_{2}}{(1-t_{1})t_{2}}.
\end{align*}

Through the duality theorem $\zeta(\{1\}^{a},b+2) = \zeta(\{1\}^{b},a+2)$, we see that 
$Z_-(2m+1) = 0$. Therefore, we combine these with Eq.\,(\ref{eq.a.1}), then we have

\begin{proposition} \label{prop.22}\cite[Proposition 3.1]{CE2021}
For any nonnegative integer $m$, we have $Z_-(2m+1)=0$ and 
\[
  Z_-(2m)
  = \zeta^{\star}(\{2\}^{m+1})
  = 2 \left( 1 - \frac{1}{2^{2m+1}} \right) \zeta(2m+2)=-2\zeta(\ov{2m+2}).
\]
\end{proposition}
The sum of multiple zeta-star values $Z^\star_+(n)$
appeared as the principal term of the evaluation of $\zeta^{\star}(3,\{2\}^{n})$
(see \cite{CE2018}):
\[
  \zeta^{\star}(3,\{2\}^{n})
  = Z_+^\star(2n+1)
    - 2 \sum_{a+b=n} \zeta(\{2\}^{a}) \zeta(2b+3).
\]

Arakawa and Kaneko \cite{AK1999} defined the function
\[
\xi_k(s)=\frac1{\Gamma(s)}\int^\infty_0\frac{t^{s-1}}{e^t-1}
\Li_k(1-e^{-t})\,dt,
\]
where $\Li_k(s)$ denotes the $k$-th polylogarithm 
$\Li_k(s)=\sum^\infty_{n=1}\frac{s^n}{n^k}$.
It is exactly the multiple zeta-star values of height one
\[
\xi_k(s)=\zeta^\star(\{1\}^{s-1},k+1).
\]
Many properties of the generalized Arakawa-Kaneko zeta functions have been 
discovered recently (ref. \cite{Chen2019, CC2010, Kargein2020}). 

Indeed, $Z^\star_+(n)$ has the generating function
\[
  \sum_{n=0}^{\infty} Z^\star_+(n) (-x)^{n}
  = \int_{0 < t_{1} < t_{2} < 1} (1-t_{2})^{x}
    \left( \frac{t_{1}}{t_{2}} \right)^{x}
    \frac{dt_{1} dt_{2}}{(1-t_{1})t_{2}}.
\]
The dual of the above is
\[
  \int_{0 < u_{1} < u_{2} < 1} \left( \frac{1-u_{2}}{1-u_{1}} \right)^{x}
    u_{1}^{x} \frac{du_{1} du_{2}}{(1-u_{1})u_{2}},
\]
which can be evaluated as
\[
  \sum_{k=1}^{\infty} \frac{1}{(k+x) (k+2x)}
    \frac{\Gamma(k+x)^{2}}{\Gamma(k) \Gamma(k+2x)}.
\]
Fortunately, we have the identity
\[
  \sum_{k=1}^{\infty} \frac{1}{(k+x) (k+2x)}
    \frac{\Gamma(k+x)^{2}}{\Gamma(k) \Gamma(k+2x)}
  = \sum_{\ell=1}^{\infty} \frac{2(-1)^{\ell-1}}{(\ell+x)^{2}}
\]
by investigating possible poles of both sides of the meromorphic functions. This leads to the evaluation of $Z^\star_+(n)$:

\begin{equation}\label{eq.zstar}
\sum_{a+b=w-2}\zeta^\star(\{1\}^a,b+2)
=2(w-1)(1-2^{1-w})\zeta(w)=-2(w-1)\zeta(\ov{w}),
\end{equation}
where $w>1$ is an integer. This formula was first proved by 
Ohno \cite[Theorem 8]{Ohno2005} in 2005.

\section{The Main Integral $J_p$ and its $\mathbb J(1), \mathbb J(2)$ Parts}\label{sec.3}
We begin with the integral
\[
  J_p
  = \frac{1}{p!} \int_{E_{2} \times E_{2}}
    \left( \log \frac{1-t_{1}}{1-u_{2}} + \log \frac{t_{2}}{t_{1}}
    + \log \frac{u_{2}}{u_{1}} \right)^{p} \frac{dt_{1} dt_{2}}{(1-t_{1})t_{2}} \frac{du_{1} du_{2}}{(1-u_{1})u_{2}},
\]
where $E_{2} \times E_{2} = \{ (t_{1}, t_{2}, u_{1}, u_{2}) \in \mathbb{R}^{4} \mid 0 < t_{1} < t_{2} < 1, 0 < u_{1} < u_{2} < 1 \}$. As

\begin{align*}
  &\quad \frac{1}{p!} \left( \log \frac{1-t_{1}}{1-u_{2}}
    + \log \frac{t_{2}}{t_{1}} + \log \frac{u_{2}}{u_{1}} \right)^{p} \\
  &= \frac{1}{p!} \left\{ \left( -\log \frac{1}{1-t_{1}}
    + \log \frac{t_{2}}{t_{1}} \right)
    + \left( \log \frac{1}{1-u_{2}} + \log \frac{u_{2}}{u_{1}} \right) \right\}^p \\
  &= \sum_{m+n=p} \frac{1}{m!n!}
    \left( -\log \frac{1}{1-t_{1}} + \log \frac{t_{2}}{t_{1}} \right)^{m} 
    \left( \log \frac{1}{1-u_{2}} + \log \frac{u_{2}}{u_{1}} \right)^{n},
\end{align*}

\noindent so we have the evaluation
\[
  J_p = \sum_{m+n=p} (-1)^{m} Z_-(m) Z_+^\star(n)
\]
where $Z_-(m)$ and $Z_+^\star(n)$ are defined in Eq\,(\ref{eq.zz}).
By Proposition \ref{prop.22} $Z_-(2m+1)=0$, thus 
\[
J_p=\sum_{m=0}^{[p/2]} Z_-(2m) Z_+^\star(p-2m)=\sum_{m+n=p}Z_-(m)Z_+^\star(n).
\]
We got an expression of the finite convolution of $Z_-(m)$
and $Z_+^\star(n)$ in \cite[Corollary 5.4]{CE2021}:
\[
J_{p}=2((-1)^p-1)\zeta(p+2,\ov{2})+(p+2)(p+1+2^{-p-2})\zeta(p+4).
\]

In this paper, we consider another expression of $J_p$.
When we decompose $E_{2} \times E_{2}$ into 6 simplices of dimension 4:

\begin{gather*}
  D_{1}: 0 < t_{1} < t_{2} < u_{1} < u_{2} < 1, \quad
  D_{2}: 0 < u_{1} < u_{2} < t_{1} < t_{2} < 1, \\
  D_{3}: 0 < t_{1} < u_{1} < t_{2} < u_{2} < 1, \quad
  D_{4}: 0 < t_{1} < u_{1} < u_{2} < t_{2} < 1, \\
  D_{5}: 0 < u_{1} < t_{1} < u_{2} < t_{2} < 1, \quad
  D_{6}: 0 < u_{1} < t_{1} < t_{2} < u_{2} < 1
\end{gather*}

\noindent and let
\[
  \mathbb J(j)
  = \frac{1}{p!} \int_{D_{j}} \left( \log \frac{1-t_{1}}{1-u_{2}}
    + \log \frac{t_{2}}{t_{1}} + \log \frac{u_{2}}{u_{1}} \right)^{p} \frac{dt_{1} dt_{2}}{(1-t_{1})t_{2}} \frac{du_{1} du_{2}}{(1-u_{1})u_{2}}.
\]
Then $J_p$ is breaked into 6 parts:
\[
  J_{p} = \sum_{j=1}^{6} \mathbb J(j).
\]
In the following, we will evaluate $\mathbb J(j)$, for $1\leq j\leq 6$, in terms of multiple zeta values.
\begin{proposition}
Given any nonnegative integer $p$, we have
\[
  \mathbb J(1)
  = \sum_{m+n=p} \sum_{a+b+c=m} \sum_{u+v=n}
    \sum_{\substack{|\boldsymbol{\alpha}|=a+u+1 \\ |\boldsymbol{\beta}|=c+v+1}} \zeta(\alpha_{0}, \alpha_{1}, \ldots, \alpha_{a}+1, \{1\}^{b}, \beta_{0}, \ldots, \beta_{c}+1).
\]
\end{proposition}
\begin{proof}
We first expand the integrand of $\mathbb J(1)$ as 
\[
\sum_{m+n=p}\frac1{m!n!}\left(\log\frac{1-t_1}{1-u_2}\right)^m
\left(\log\frac{t_2}{t_1}+\log\frac{u_2}{u_1}\right)^n.
\]
Since $\mathbb J(1)$ is an integral on $D_1: 0<t_1<t_2<u_1<u_2<1$, we replace the factor
\[
  \log \frac{1-t_{1}}{1-u_{2}}
\]
by its equal
\[
  \log \frac{1-t_{1}}{1-t_{2}} + \log \frac{1-t_{2}}{1-u_{1}}
    + \log \frac{1-u_{1}}{1-u_{2}}.
\]
Then we expand the integrand of $\mathbb J(1)$ as

\begin{align*}
  &\sum_{m+n=p} \sum_{a+b+c=m} \sum_{u+v=n} \frac{1}{a!b!c!u!v!} \\
  &\quad \times \left( \log \frac{1-t_{1}}{1-t_{2}} \right)^{a}
    \left( \log \frac{1-t_{2}}{1-u_{1}} \right)^{b}
    \left( \log \frac{1-u_{1}}{1-u_{2}} \right)^{c}
    \left( \log \frac{t_{2}}{t_{1}} \right)^{u}
    \left( \log \frac{u_{2}}{u_{1}} \right)^{v}.
\end{align*}

The factor
\[
  \frac{1}{a!u!} \left( \log \frac{1-t_{1}}{1-t_{2}} \right)^{a}
    \left( \log \frac{t_{2}}{t_{1}} \right)^{u}
\]
forms a sum and the other factor
\[
  \frac{1}{c!v!} \left( \log \frac{1-u_{1}}{1-u_{2}} \right)^{c}
    \left( \log \frac{u_{2}}{u_{1}} \right)^{v}
\]
forms another sum, so
\[
  \mathbb J(1)
  = \sum_{m+n=p} \sum_{a+b+c=m} \sum_{u+v=n}
    \sum_{\substack{|\boldsymbol{\alpha}|=a+u+1 \\ |\boldsymbol{\beta}|=c+v+1}} \zeta(\alpha_{0}, \alpha_{1}, \ldots, \alpha_{a}+1, \{1\}^{b}, \beta_{0}, \ldots, \beta_{c}+1). \qedhere
\]
\end{proof}

\begin{proposition}
For any nonnegative integer $p$, we have
\[
  \mathbb J(2) = \sum_{c+d+m=p} (-1)^{m} \zeta(c+2,\{1\}^{m},d+2).
\]
\end{proposition}
\begin{proof}
On $D_{2}: 0 < u_{1} < u_{2} < t_{1} < t_{2} < 1$, the factor
\[
  \log \frac{1-t_{1}}{1-u_{2}}
\]
is replaced by
\[
  -\log \frac{1-u_{2}}{1-t_{1}}
\]
and the integrand
\[
  \frac{1}{p!} \left( -\log \frac{1-u_{2}}{1-t_{1}} + \log \frac{t_{2}}{t_{1}}
    + \log \frac{u_{2}}{u_{1}} \right)^p
\]
is expanded into
\[
  \sum_{c+d+m=p} \frac{(-1)^{m}}{c!\,d!\,m!}
    \left( \log \frac{1-u_{2}}{1-t_{1}} \right)^{m}
    \left( \log \frac{u_{2}}{u_{1}} \right)^{c}
    \left( \log \frac{t_{2}}{t_{1}} \right)^{d}.
\]
So that
\[
  \mathbb J(2) = \sum_{c+d+m=p} (-1)^m\zeta(c+2,\{1\}^{m},d+2). \qedhere
\]
\end{proof}
It is noting that we have another expression of $\mathbb J(2)$ in \cite[Theorem 3]{CE2022}:
\[
\sum_{a+b+m=p}(-1)^m\zeta(a+2,\{1\}^m,b+2)=2[(-1)^p-1]\zeta(p+2,\ov{2})
+\left(1-2^{-p-2}\right)\zeta(p+4).
\]
%
\section{The Integrals $\mathbb J(3), \mathbb J(4), \mathbb J(5)$, and $\mathbb J(6)$ Parts}
\label{sec.4}
%
In this section, we will give the evaluations of $\mathbb J(j)$ for $3\leq j\leq 6$.
These expressions give us the part of the weighted sum of our main theorem.
Our method of proving the following proposition is inspired by the 
approach used in \cite[Theorem 5.3]{OEL2013}.

\begin{proposition}
Let $p$ be a nonnegative integer. We have
\[
  \mathbb J(3)
  = \sum_{m+n=p} \sum_{|\boldsymbol{\alpha}|=p+3}
    \zeta(\alpha_{0}, \alpha_{1}, \ldots, \alpha_{m}, \alpha_{m+1}+1) \sum_{a+b+c=m} W(a,b,c)
\]
with $\sigma(r) = \alpha_{0} + \alpha_{1} + \cdots + \alpha_{r}$ and
\begin{equation}\label{eq.w}
  W(a,b,c)
  = 2^{\sigma(a+b+1)-\sigma(a)-(b+1)} (1-2^{1-\alpha_{a+b+1}}).
\end{equation}
\end{proposition}
\begin{proof}
The integral $\mathbb J(3)$ is on $D_3: 0<t_1<u_1<t_2<u_2<1$.
Therefore we replace the integrand of $\mathbb J(3)$ by
\[
  \frac{1}{p!} \left( \log \frac{1-t_{1}}{1-u_{1}}
    + \log \frac{1-u_{1}}{1-t_{2}} + \log \frac{1-t_{2}}{1-u_{2}}
    + \log \frac{u_{1}}{t_{1}} + 2\log \frac{t_{2}}{u_{1}}
    + \log \frac{u_{2}}{t_{2}} \right)^{p}.
\]
Then we expand it as
\[
  \sum_{\substack{m+n=p\\ a+b+c=m\\ u+v+w=n}} \frac{2^{v}}{a!b!c!u!v!w!} 
  \left( \log \frac{1-t_{1}}{1-u_{1}} \right)^{a}
    \left( \log \frac{1-u_{1}}{1-t_{2}} \right)^{b}
    \left( \log \frac{1-t_{2}}{1-u_{2}} \right)^{c}
    \left( \log \frac{u_{1}}{t_{1}} \right)^{u}
    \left( \log \frac{t_{2}}{u_{1}} \right)^{v}
    \left( \log \frac{u_{2}}{t_{2}} \right)^{w}.
\]
This yields
\[
  \mathbb J(3)
  = \sum_{\substack{m+n=p\\ a+b+c=m\\u+v+w=n}} 2^{v}
    \sum_{\substack{|\boldsymbol{\alpha}|=a+u+1 \\ |\boldsymbol{\beta}|=b+v+1 \\ |\boldsymbol{\gamma}|=c+w+1}}
    \zeta(\alpha_{0}, \alpha_{1}, \ldots, \alpha_{a}, \beta_{0}, \beta_{1}, \ldots, \beta_{b}+\gamma_{0}, \gamma_{1}, \ldots, \gamma_{c}+1).
\]
We change the variables $\bm{\alpha,\beta,\gamma}$ to a nonnegative vector variable
$\bm e=(e_0,e_1,\ldots,e_{m+1})$. Then 

\begin{align*}
\mathbb J(3)&=\sum_{\substack{m+n=p\\a+b+c=m\\|\bm e|=n}}
\left(\sum_{i=0}^{e_{a+b+1}}2^{e_{a+b+1}-i}\right)
2^{e_{a+1}+e_{a+2}+\cdots+e_{a+b}} \\
&\qquad\qquad\qquad\times
\zeta(e_0+1,\ldots,e_{a+b}+1,e_{a+b+1}+2,e_{a+b+2}+1,\ldots,e_{m}+1,e_{m+1}+2).
\end{align*}

\noindent Since 
\[
\sum_{i=0}^{e_{a+b+1}}2^{e_{a+b+1}-i}=2^{e_{a+b+1}+1}-1,
\]
we can change the nonnegative vector variable $\bm e$ to a new positive variable $\bm \alpha$
and $\mathbb J(3)$ becomes
\[
\sum_{\substack{m+n=p\\a+b+c=m\\|\bm\alpha|=m+n+3}}
2^{\alpha_{a+1}+\cdots+\alpha_{a+b}-b}\left(2^{\alpha_{a+b+1}-1}-1\right)
\zeta(\alpha_0,\ldots,\alpha_m,\alpha_{m+1}+1).
\]
We rewrite $\mathbb J(3)$ as
\[
  \sum_{m+n=p} \sum_{|\boldsymbol{\alpha}|=p+3}
    \zeta(\alpha_{0}, \alpha_{1}, \ldots, \alpha_{m}, \alpha_{m+1}+1) \sum_{a+b+c=m} 
    W_{\bm\alpha}(a,b,c)
\]
with $\sigma(r) = \alpha_{0} + \alpha_{1} + \cdots + \alpha_{r}$ and
\[
  W_{\bm\alpha}(a,b,c)
  = 2^{\sigma(a+b+1)-\sigma(a)-(b+1)} (1-2^{1-\alpha_{a+b+1}}). \qedhere
\]
\end{proof}

In the same manner as the previous proposition, we obtain the following.

\begin{proposition}
Notation as introduced above, we have
\[
  \mathbb J(4)
  = \sum_{m+n=p} \sum_{|\boldsymbol{\alpha}|=p+3}
    \zeta(\alpha_{0}, \alpha_{1}, \ldots, \alpha_{m}, \alpha_{m+1}+1) \sum_{a+b=m} 
    W_{\bm\alpha}(a,b,0),
\]
where $W_{\bm\alpha}(a,b,c)$ is defined in Eq.\,(\ref{eq.w}) and 
\[
  W_{\bm\alpha}(a,b,0)
  = 2^{\sigma(a+b+1)-\sigma(a)-(b+1)} (1-2^{1-\alpha_{a+b+1}}).
\]
\end{proposition}

\begin{proposition}
Notation as introduced above, then we have
\[
  \mathbb J(5)
  = \sum_{m+n=p} \sum_{|\boldsymbol{\alpha}|=p+3}
    \zeta(\alpha_{0}, \alpha_{1}, \ldots, \alpha_{m}, \alpha_{m+1}+1) 
    W_{\bm\alpha}(0,m,0),
\]
where 
\[
  W_{\bm\alpha}(0,m,0)
  = 2^{\sigma(m+1)-\sigma(0)-(m+1)} (1-2^{1-\alpha_{m+1}})
  = 2^{n+2-\sigma(0)} (1-2^{1-\alpha_{m+1}}).
\]
\end{proposition}

\begin{proposition}
Notation as introduced above, then we have
\[
  \mathbb J(6)
  = \sum_{m+n=p} \sum_{|\boldsymbol{\alpha}|=p+3}
    \zeta(\alpha_{0}, \alpha_{1}, \ldots, \alpha_{m}, \alpha_{m+1}+1) \sum_{b+c=m} 
    W_{\bm\alpha}(0,b,c)
\]
with $\sigma(r) = \alpha_{0} + \alpha_{1} + \cdots + \alpha_{r}$ and
\[
  W_{\bm\alpha}(0,b,c)
  = 2^{\sigma(b+1)-\sigma(0)-(b+1)} (1-2^{1-\alpha_{b+1}}).
\]
\end{proposition}

So up to now, our shuffle relation appeared to be the form

\begin{align*}
  J_{p}
  &= \mathbb J(1) + \mathbb J(2) \\
  &\quad + \sum_{m+n=p} \sum_{|\boldsymbol{\alpha}|=p+3}
    \zeta(\alpha_{0}, \alpha_{1}, \ldots, \alpha_{m}, \alpha_{m+1}+1) \\
  &\qquad \times\left( \sum_{a+b+c=m} W_{\bm\alpha}(a,b,c) 
  +\sum_{a+b=m}W_{\bm\alpha}(a,b,0) + W_{\bm\alpha}(0,m,0) 
  + \sum_{b+c=m}W_{\bm\alpha}(0,b,c)  \right).
\end{align*}

We write it in a compact form:

\begin{align}\label{eq.j12w}
  J_{p}\nonumber
  &= \mathbb J(1) + \mathbb J(2) \\ 
  &\quad + \sum_{m+n=p} \sum_{|\boldsymbol{\alpha}|=p+3}
    \zeta(\alpha_{0}, \alpha_{1}, \ldots, \alpha_{m}, \alpha_{m+1}+1) \\
  &\qquad \times \sum_{a+b+c=m} \Bigl\{ W_{\bm\alpha}(a,b,c)  \nonumber
  + W_{\bm\alpha}(a,b,c=0) + W_{\bm\alpha}(a=0,b=m,c=0) + W_{\bm\alpha}(a=0,b,c)  \Bigr\}.
\end{align}

\section{Another Way to Evaluate $\mathbb J(1)$ and $\mathbb J(2)$ Parts}\label{sec.5}
The modified Bell polynomials are defined by \cite{Chen2017, CC2010}
\[
  \exp \left\{ \sum_{k=1}^{\infty} \frac{x_{k} z^{k}}{k} \right\}
  = \sum_{m=0}^{\infty} P_{m}(x_{1},x_{2},\ldots,x_{m}) z^{m}.
\]
So that
\[
  P_{m}(x_{1},x_{2},\ldots,x_{m})
  = \sum_{k_{1}+2k_{2}+\cdots+mk_{m}=m} \frac{1}{k_{1}! k_{2}! \cdots k_{m}!}
    \left( \frac{x_{1}}{1} \right)^{k_{1}}
    \left( \frac{x_{2}}{2} \right)^{k_{2}} \cdots
    \left( \frac{x_{m}}{m} \right)^{k_{m}}.
\]
In particular, for $m = 0,1,2,3$,
\[
  P_{0} = 1, \quad
  P_{1}(x_{1}) = x_{1}, \quad
  P_{2}(x_{1},x_{2}) = \frac{1}{2} (x_{1}^{2}+x_{2}), \quad
  P_{3}(x_{1},x_{2},x_{3}) = \frac{1}{3} (x_{1}^{3}+3x_{1}x_{2}+2x_{3}).
\]
It is well-known that
\[
  P_{n}(\zeta(2), -\zeta(4), \ldots, (-1)^{n+1} \zeta(2n)) = \zeta(\{2\}^{n})
  \quad \textrm{and} \quad
  P_{n}(\zeta(2), \zeta(4), \ldots, \zeta(2n)) = \zeta^{\star}(\{2\}^{n}).
\]
In order to give another expression of $\mathbb J(1)$ we need the following lemma.
\begin{lemma}\cite[Proposition 4.4]{CCE2016}\label{lma.dd}
For a pair of positive integers $k_{1}$, $k_{2}$ with $k_{1} \leq k_{2}$, let
\[
  Q(y)
  = \frac{\Gamma(k_{1}-y)}{\Gamma(k_{1})}
    \frac{\Gamma(k_{2}+1)}{\Gamma(k_{2}+1-y)}.
\]
Then for any nonnegative integer $n$,
\[
  \frac{Q^{(n)}(0)}{n!}
  = P_{n}(h_{1},h_{2},\ldots,h_{n})
  = \sum_{k_{1} \leq \ell_{1} \leq \ell_{2} \leq \cdots \leq \ell_{n}
    \leq k_{2}} \frac{1}{\ell_{1} \ell_{2} \cdots \ell_{n}}
\]
with $h_{n} = \sum_{j=k_{1}}^{k_{2}} \frac{1}{j^{n}}$.
\end{lemma}
\begin{proposition}
Let $p$ be a nonnegative integer and 
\[
  \mathbb J(1)
  = \sum_{m+n=p} \frac{1}{m!n!} \int_{0 < t_{1} < t_{2} < u_{1} < u_{2} < 1}
    \left( \log \frac{1-t_{1}}{1-u_{2}} \right)^{m}
    \left( \log \frac{t_{2}}{t_{1}} + \log \frac{u_{2}}{u_{1}} \right)^{n} \frac{dt_{1} dt_{2}}{(1-t_{1})t_{2}} \frac{du_{1} du_{2}}{(1-u_{1})u_{2}}.
\]
Then

\begin{align*}
  \mathbb J(1)
  &= \sum_{m+n=p}\sum_{c+d=n} \{ \zeta^{\star}(c+2,\{1\}^{m},d+2) - \zeta(m+n+4) \} \\
  &= \left\{ \frac{p(p+3)}{2} + \frac{p+3}{2^{p+2}} \right\} \zeta(p+4).
\end{align*}

\end{proposition}
\begin{proof}
Let $S(m,n)$ be the general term inside the summation of $\mathbb J(1)$. Then

\begin{align*}
  \sum_{m=0}^{\infty} \sum_{n=0}^{\infty} S(m,n) x^{m} y^{n}
  &= \int_{0 < t_{1} < t_{2} < u_{1} < u_{2} < 1}
    \left( \frac{1-t_{1}}{1-u_{2}} \right)^{x}
    \left( \frac{t_{2}}{t_{1}} \right)^{y} \left( \frac{u_{2}}{u_{1}} \right)^{y} \frac{dt_{1} dt_{2}}{(1-t_{1})t_{2}}
    \frac{du_{1} du_{2}}{(1-u_{1})u_{2}} \\
  &\equiv G({x,y})
\end{align*}

\noindent is the generating function for the double sequence $\{S(m,n)\}$ and
\[
  S(m,n)
  = \frac{1}{m!n!} \left( \frac{\partial}{\partial x} \right)^{m}
    \left( \frac{\partial}{\partial y} \right)^{n} G(x,y) \Big|_{x=y=0}.
\]
Beginning with
\[
  \frac{1}{(1-t_{1})^{1-x}}
  = \sum_{k=1}^{\infty} \frac{\Gamma(k-x)}{\Gamma(k) \Gamma(1-x)} t_{1}^{k-1},
\]
and then integrating with respect to $t_{1}$, $t_{2}$ and $u_{1}$, we obtain that
\[
  G(x,y)
  = \sum_{k=1}^{\infty} \sum_{\ell=1}^{\infty} \frac{1}{(k-y)k(k+\ell-y)}
    \frac{\Gamma(k-x)}{\Gamma(k) \Gamma(1-x)} \int_{0}^{1} (1-u_{2})^{-x} u_{2}^{k+\ell-1} \, du_{2}.
\]
The value of integral is
\[
  \frac{\Gamma(1-x) \Gamma(k+\ell)}{\Gamma(k+\ell+1-x)},
  \quad \textrm{or} \quad
  \frac{1}{k+\ell} \frac{\Gamma(1-x) \Gamma(k+\ell+1)}{\Gamma(k+\ell+1-x)}.
\]
So
\[
  G(x,y)
  = \sum_{k=1}^{\infty} \sum_{\ell=1}^{\infty}
    \frac{1}{k(k-y)(k+\ell)(k+\ell-y)}
    \frac{\Gamma(k-x) \Gamma(k+\ell+1)}{\Gamma(k) \Gamma(k+\ell+1-x)}
\]
after differentiations with respect to $x$, $y$ for $m$, $n$ times, and we use Lemma \ref{lma.dd},
\[
  S(m,n)
  = \sum_{c+d=n} \left(\sum_{k=1}^{\infty} \sum_{\ell=1}^{\infty}
    \frac{1}{k^{c+2} (k+\ell)^{d+2}}
    \sum_{k \leq \ell_{1} \leq \cdots \leq \ell_{m} \leq k+\ell} 
    \frac{1}{\ell_{1} \ell_{2} \cdots \ell_{m}}\right).
\]
The general term of $S(m,n)$ is
\[
  \zeta^{\star}(c+2,\{1\}^{m},d+2) - \zeta(m+n+4)
\]
and hence
\[
  \mathbb J(1)
  = \sum_{m+n=p} \sum_{c+d=n}
    \{ \zeta^{\star}(c+2,\{1\}^{m},d+2) - \zeta(m+n+4) \}.
\]

We use a result in \cite[Theorem 2]{CE2022}:
\[
\sum_{a+b+m=p}\zeta^\star(a+2,\{1\}^m,b+2)
=\left(p^2+3p+1+\frac{p+3}{2^{p+2}}\right)\zeta(p+4).
\]
Therefore we have 
\[
  \mathbb J(1)
  = \left( \frac{p(p+3)}{2} + \frac{p+3}{2^{p+2}} \right) \zeta(p+4).
\]
\end{proof}

To transform
\[
 \mathbb J(2)
  = \sum_{c+d+m=p} (-1)^{m}\zeta(c+2,\{1\}^{m},d+2)
\]
into multiple zeta values related to $\mathbb J(1)$, we need the following reflection formula.

\begin{proposition}\cite[Proposition 4]{CE2022}\label{prop.reflection}
For an $r$-tuple $\bm\alpha=(\alpha_1,\alpha_2,\ldots,\alpha_r)$ of positive integers
with $\alpha_1\geq 2$, $\alpha_r\geq 2$, we have
\begin{equation}\label{eq.reflection}
\zeta(\alpha_1,\ldots,\alpha_r)+(-1)^r\zeta^\star(\alpha_r,\ldots,\alpha_1)
=\sum^{r-1}_{k=1}(-1)^{k+1}\zeta^\star(\alpha_k,\ldots,\alpha_1)
\zeta(\alpha_{k+1},\ldots,\alpha_r).
\end{equation}
\end{proposition}

\begin{proposition}
Let $p$ be a nonnegative integer and 
\[
  \mathbb J(2) = \sum_{c+d+m=p} (-1)^{m} \zeta(c+2,\{1\}^{m},d+2).
\]
Then we have
\[
  \mathbb J(2)
  = J_{p} - \sum_{c+d+m=p} \zeta^{\star}(d+2,\{1\}^{m},c+2).
\]
\end{proposition}
\begin{proof}
By the reflection formula (see Proposition \ref{prop.reflection}), we have
\[
  \zeta(c+2,\{1\}^{m},d+2)
  = \sum_{a+b=m} (-1)^{a} \zeta^{\star}(\{1\}^{a},c+2) \zeta(\{1\}^{b},d+2) 
    + (-1)^{m+1} \zeta^{\star}(d+2,\{1\}^{m},c+2),
\]
so that

\begin{align*}
  \mathbb J(2)
  &= \sum_{m+n=p}\left( (-1)^{m} \sum_{a+b=m} (-1)^{a} \sum_{c+d=n} 
    \zeta^{\star}(\{1\}^{a},c+2) \zeta(\{1\}^{b},d+2) 
  - \sum_{c+d=n} \zeta^{\star}(d+2,\{1\}^{m},c+2) \right)\\
  &= \sum_{a+b+c+d=p} (-1)^{b} \zeta^{\star}(\{1\}^{a},c+2) 
    \zeta(\{1\}^{b},d+2) - \sum_{c+d+m=n} \zeta^{\star}(d+2,\{1\}^{m},c+2).
\end{align*}

The first sum can be rewritten as
\[
  \sum_{m+n=p} \sum_{b+d=m} (-1)^{b} \zeta(\{1\}^{b},d+2) \sum_{a+c=n} 
    \zeta^{\star}(\{1\}^{a},c+2),
\]
which is precisely equal to $J_{p}$ expressed as
\[
  \sum_{m+n=p} \frac{1}{m!n!} \int_{E_{2}} 
    \left( -\log \frac{1}{1-t_{1}} + \log \frac{t_{2}}{t_{1}} \right)^{m} \frac{dt_{1} dt_{2}}{(1-t_{1})t_{2}} \int_{E_{2}} 
    \left( \log \frac{1}{1-u_{2}} + \log \frac{u_{2}}{u_{1}} \right)^{n} \frac{du_{1} du_{2}}{(1-u_{1})u_{2}}.
\]
\end{proof}

According to previous propositions concerning $\mathbb J(1)$ and $\mathbb J(2)$, 
we have the following

\begin{corollary}
Notation as above, then we have
\[
  \mathbb J(1) + \mathbb J(2) = J_{p} - \frac{(p+1)(p+2)}{2} \zeta(p+4).
\]
\end{corollary}

Now we combine Eq.\,(\ref{eq.j12w}) and we conclude our main theorem.
%
\section{An Application of $\mathbb J(3)+\mathbb J(4)+\mathbb J(5)+\mathbb J(6)$}\label{sec.6}
%
Let us begin with another integral 
\[
\mathbb I(p)
=\frac1{p!}\int_{E_2\times E_2}\left(\log\frac{1-t_1}{1-t_2}
+\log\frac{t_2}{t_1}+\log\frac{u_2}{u_1}\right)^p
\frac{dt_1dt_2}{(1-t_1)t_2}\frac{du_1du_2}{(1-u_1)u_2},
\]
where $p$ is a nonnegative integer.
Since 
\[
\frac1{m!}\int_{0<t_1<t_2<1}\left(\log\frac{1-t_1}{1-t_2}+\log\frac{t_2}{t_1}\right)^m
\frac{dt_1}{1-t_1}\frac{dt_2}{t_2}
=\zeta^\star(\{1\}^m,2)=(m+1)\zeta(m+2),
\]
The integral $\mathbb I(p)$ can be written as 
\[
\sum_{m+n=p}(m+1)\zeta(m+2)\zeta(n+2)\quad\mbox{or}\quad
\frac{p+2}2\sum_{m+n=p}\zeta(m+2)\zeta(n+2).
\]
On the other hand, we decomposed the integral $\mathbb I(p)$ into six parts
in a similar way to $\mathbb J_p$:
\[
\mathbb I(p)=\sum^6_{j=1}\mathbb I(p;j),
\]
where
\[
\mathbb I(p;j)=\frac1{p!}\int_{D_j}
\left(\log\frac{1-t_1}{1-t_2}
+\log\frac{t_2}{t_1}+\log\frac{u_2}{u_1}\right)^p
\frac{dt_1dt_2}{(1-t_1)t_2}\frac{du_1du_2}{(1-u_1)u_2},
\]
The first integral $\mathbb I(p;1)$ is on $D_1: 0<t_1<t_2<u_1<u_2<1$, and we have
\[
\mathbb I(p;1)=\sum_{a+b+c=p}
\sum_{|\bm \alpha|=a+b+1}\zeta(\alpha_0,\ldots,\alpha_{a-1},\alpha_a+1,c+2).
\]
The integral $\mathbb I(p;2)$ is 
\[
\mathbb I(p;2)=\sum_{a+b+c=p}
\sum_{|\bm \beta|=b+c+1}\zeta(a+2,\beta_0,\ldots,\beta_{b-1},\beta_b+1).
\]
The next four parts have corresponding equality with 
the parts $\mathbb J(i)$, for $3\leq i\leq 6$, as
\[
\mathbb I(p;3)=\mathbb J(4),\quad
\mathbb J(p;4)=\mathbb J(3),\quad
\mathbb J(p;5)=\mathbb J(6),\quad
\mathbb J(p;6)=\mathbb J(5).
\]
Therefore, we get
\[
\sum_{j=3}^6\mathbb I(p;j)=\sum^6_{j=3}\mathbb J(j)=\frac{(p+1)(p+2)}2\zeta(p+4).
\]
This gives us the following theorem.

\begin{theorem}
Let $p$ be an nonnegative integer. Then

\begin{align*}
&\sum_{a+b+c=p}
\left( \sum_{|\bm \alpha|=a+b+1}\zeta(\alpha_0,\ldots,\alpha_{a-1},\alpha_a+1,c+2)
+\sum_{|\bm \beta|=b+c+1}\zeta(a+2,\beta_0,\ldots,\beta_{b-1},\beta_b+1)\right)\\
&\qquad\qquad=\frac{p+2}2\sum_{m+n=p}\zeta(m+2)\zeta(n+2)-\frac{(p+1)(p+2)}2\zeta(p+4).
\end{align*}
\end{theorem}

Here we also give another proof using a duality theorem due to Ohno \cite{Ohno1999}.
Given any nonnegative integer $a$, we know that the dual of $\zeta(\{1\}^a,2,2)$
is $\zeta(2,a+2)$ and the dual of $\zeta(2,\{1\}^a,2)$ is $\zeta(a+2,2)$, respectively.
Then for a nonnegative integer $m$, we have 

\begin{align*}
\sum_{|\bm c|=m}\zeta(c_1+1,\ldots,c_a+1,c_{a+1}+2,c_{a+2}+2)
&=\sum_{|\bm d|=m}\zeta(d_1+2,d_2+a+2),\\
\sum_{|\bm c|=m}\zeta(c_1+2,c_2+1,\ldots,c_{a+1}+1,c_{a+2}+2)
&=\sum_{|\bm d|=m}\zeta(d_1+a+2,d_2+2).
\end{align*}
We substitue the above formulas into $\mathbb I(p;1)+\mathbb I(p;2)$,
then the desired result will be obtained.
\section{Another Expression of $\mathbb J(3)$}\label{sec.7}
The sum of multiple zeta values
\[
  \mathbb J(3)
  = \sum_{m+n=p} \sum_{|\boldsymbol{\alpha}|=p+3} 
    \zeta(\alpha_{0},\alpha_{1},\ldots,\alpha_{m},\alpha_{m+1}+1) \sum_{a+b+c=m} W(a,b,c)
\]
came from a joint of three sums
\[
  \sum_{m+n=p} \sum_{a+b+c=m} \sum_{u+v+w=n} 2^{v} 
    \sum_{\substack{|\boldsymbol{\alpha}|=a+u+1 \\ |\boldsymbol{\beta}|=b+v+1 \\ |\boldsymbol{\gamma}|=c+w+1}} 
    \zeta(\alpha_{0}, \alpha_{1}, \ldots, \alpha_{a}, \beta_{0}, \beta_{1}, \ldots, \beta_{b}+\gamma_{0}, \ldots, \gamma_{c}+1)
\]
and it has the integral representation
\[
  \frac{1}{p!} \int_{D_{3}} \left( \log \frac{1-t_{1}}{1-u_{2}} 
    + \log \frac{t_{2}}{t_{1}} + \log \frac{u_{2}}{u_{1}} \right)^{p} \frac{dt_{1} dt_{2}}{(1-t_{1})t_{2}} \frac{du_{1} du_{2}}{(1-u_{1})u_{2}},
\]
which is equal to
\[
  \sum_{m+n=p} \frac{1}{m!n!} \int_{D_{3}} 
    \left( \log \frac{1-t_{1}}{1-u_{2}} \right)^{m} 
    \left( \log \frac{t_{2}}{t_{1}} + \log \frac{u_{2}}{t_{1}} \right)^{n} \frac{dt_{1} dt_{2}}{(1-t_{1})t_{2}} \frac{du_{1} du_{2}}{(1-u_{1})u_{2}}.
\]
Let $T(m,n)$ be the general term in the above sum. Here we are going to find the value of $T(m,n)$ through its generating function.

\begin{proposition}
For integers $m,n \geq 0$, let
\[
  T(m,n)
  = \frac{1}{m!n!} \int_{D_{3}} \left( \log \frac{1-t_{1}}{1-u_{2}} \right)^{m}
    \left( \log \frac{t_{2}}{t_{1}} + \log \frac{u_{2}}{u_{1}} \right)^{n} \frac{dt_{1} dt_{2}}{(1-t_{1})t_{2}} \frac{du_{1} du_{2}}{(1-u_{1})u_{2}}.
\]
Then
\[
  T(m,n)
  = \sum_{|\boldsymbol{\alpha}|=n+3} 
    \{ \zeta^{\star}(\alpha_{1},\{1\}^{m},\alpha_{2}+1) - \zeta(m+n+4) \} (2^{\alpha_{2}-1}-1).
\]
\end{proposition}
\begin{proof}
The generating function for the double sequence $T(m,n)$ is

\begin{align*}
  \sum_{m=0}^{\infty} \sum_{n=0}^{\infty} T(m,n) x^{m} y^{n}
  &\equiv G(x,y) \\
  &= \int_{D_{3}} \left( \frac{1-t_{1}}{1-u_{2}} \right)^{x} 
    \left( \frac{t_{2}}{t_{1}} \right)^{y} 
    \left( \frac{u_{2}}{u_{1}} \right)^{y} \frac{dt_{1} dt_{2}}{(1-t_{1})t_{2}} \frac{du_{1} du_{2}}{(1-u_{1})u_{2}}.
\end{align*}
Like the case of $\mathbb J(1)$, $G(x,y)$ can be evaluated as
\[
  \sum_{k=1}^{\infty} \sum_{\ell=1}^{\infty} 
    \frac{1}{(k-y)(k+\ell-2y)(k+\ell-y)(k+\ell)} \frac{\Gamma(k-x)}{\Gamma(k)} \frac{\Gamma(k+\ell+1)}{\Gamma(k+\ell+1-x)}
\]
Hence

\begin{align*}
  T(m,n)
  &= \frac{1}{m!n!} \left( \frac{\partial}{\partial x} \right)^{m} 
    \left( \frac{\partial}{\partial y} \right)^{n} G(x,y) \Big|_{x=y=0} \\
  &= \sum_{a+b+c=n} \sum_{k=1}^{\infty} \sum_{\ell=1}^{\infty} 
    \frac{2^{b}}{k^{a+1} (k+\ell)^{b+c+3}} 
    \sum_{k \leq k_{1} \leq k_{2} \leq \cdots \leq k_{m} \leq k+\ell} \frac{1}{k_{1} k_{2} \cdots k_{m}}.
\end{align*}
It is equal to
\[
  \sum_{|\boldsymbol{\alpha}|=n+3}
    \{ \zeta^{\star}(\alpha_{1},\{1\}^{m},\alpha_{2}+1) - \zeta(m+n+4) \} (2^{\alpha_{2}-1}-1),
\]
as asserted.
\end{proof}

\begin{corollary}
For any nonnegaitve integer $p$, we have
\[
  \mathbb J(3)
  = \sum_{m+n=p} \sum_{|\boldsymbol{\alpha}|=n+3}
    \Bigl\{ \zeta^{\star}(\alpha_{1},\{1\}^{m},\alpha_{2}+1) - \zeta(p+4) \Bigr\} 
    (2^{\alpha_{2}-1}-1).
\]
\end{corollary}

\begin{remark}
We have \cite[Theorem 1]{CE2022} a duality theorem:
\[
  \sum_{|\boldsymbol{\alpha}|=n+3} 
    \zeta^{\star}(\alpha_{1},\{1\}^{m},\alpha_{2}+1)
  = (m+n+3) \zeta(m+n+4).
\]
However, it is still unknown for the weighted sum
\[
  \sum_{|\boldsymbol{\alpha}|=n+3} 
    \zeta^{\star}(\alpha_{1},\{1\}^{m},\alpha_{2}+1) 2^{\alpha_{2}}.
\]
However, our weighted sum formula did provide an approximate evaluation of weighted sum of zeta-star values of such kind.
\end{remark}

In the next section, we will try give some weighted alternating Euler sums.
%
\section{Weighted Alternating Euler Sum Formulas}\label{sec.8}
%
The weighted alternating Euler sum formula
\[
\sum_{m+n=p}2^n\zeta(m+1,\ov{n+1})
=\frac{p+1}2\zeta(p+2)+\zeta(p+1,\ov{1})+\zeta(\ov{p+2})
\]
was obtained in \cite[Eq.\,(5.3)]{CE2021}. Now we produce a more general
formula which it covers the above.
\begin{theorem}
For any nonnegative integers $p, q$ and a real number $\lambda$, we have

\begin{align}\label{eq.81}
&\sum_{m+n=p}\binom{n}{q}\lambda^{n-q}\zeta(\ov{m+1})\zeta(\ov{n+1}) \\
&=\sum_{m+n=p}\binom{n}{q}(\lambda+1)^{n-q}\zeta(m+1,\ov{n+1})
+\sum_{\substack{m+n=p\\a+b=q}}\binom{m}{a}\binom{n}{b}
\lambda^{m-a}(\lambda+1)^{n-b}\zeta(m+1,\ov{n+1}).\nonumber
\end{align}
\end{theorem}
\begin{proof}
It is known that \cite[Eq.\,(2.3)]{CE2021}:
\[
\zeta(\ov{p+1})=\frac{-1}{p!}\int^1_0\left(\log\frac1t\right)^p\frac{dt}{1+t}.
\]
We evaluate the following sum
\[
S(p)=\sum_{m+n=p}\lambda^n\zeta(\ov{m+1})\zeta(\ov{n+1})
\]
as an integral form
\[
\frac{1}{p!}\int_{\substack{0<t_1<t_2<1\\0<t_2<t_1<1}}
\left(\log\frac1{t_1}+\lambda\log\frac1{t_2}\right)^p
\frac{dt_1dt_2}{(1+t_1)(1+t_2)}.
\]
This integral can be decomposed into two parts:

\begin{align*}
\mathbb K_1 &= \frac{1}{p!}\int_{0<t_1<t_2<1}
\left(\log\frac1{t_1}+\lambda\log\frac1{t_2}\right)^p
\frac{dt_1dt_2}{(1+t_1)(1+t_2)},\mbox{ and}\\
\mathbb K_2 &= \frac{1}{p!}\int_{0<t_2<t_1<1}
\left(\log\frac1{t_1}+\lambda\log\frac1{t_2}\right)^p
\frac{dt_1dt_2}{(1+t_1)(1+t_2)}.
\end{align*}
It can be seen that 
\[
\mathbb K_1=\sum_{a+b=p}(\lambda+1)^b\zeta(a+1,\ov{b+1})\mbox{ and }
\mathbb K_2=\sum_{a+b=p}\lambda^a(\lambda+1)^b\zeta(a+1,\ov{b+1}).
\]
Therefore we have 
\begin{equation}
\sum_{m+n=p}\lambda^n\zeta(\ov{m+1})\zeta(\ov{n+1})
=\sum_{m+n=p}\zeta(m+1,\ov{n+1})\Bigl[
(\lambda+1)^n(\lambda^m+1)\Bigr].
\end{equation}

We differential $q$ times with $\lambda$ in the above equation, we obtain
the final identity.
\end{proof}
We list some examples. Let $\lambda=0$ in Eq.\,(\ref{eq.81}). The following identity is obtained.
\[
\sum_{m+n=p}\binom{n}{q}\zeta(m+1,\ov{n+1})
=\zeta(\ov{q+1})\zeta(\ov{p+1-q})
-\sum_{a+b=q}\binom{p-a}{b}\zeta(a+1,\ov{p+1-a}).
\]

\noindent Then 

\begin{align*}
\sum_{m+n=p}\zeta(m+1,\ov{n+1})&=\zeta(\ov{1})\zeta(\ov{p+1})-\zeta(1,\ov{p+1}).
\quad\mbox{(for $q=0$, \cite[Eq.\,(2.15)]{Teo2018})}\\
\sum_{m+n=p}n\zeta(m+1,\ov{n+1})&=\zeta(\ov{2})\zeta(\ov{p})
-p\zeta(1,\ov{p+1})-\zeta(2,\ov{p}).
\quad\mbox{(for $q=1$)}\\
\sum_{m+n=p}\binom{n}{2}\zeta(m+1,\ov{n+1}) &=
\zeta(\ov{3})\zeta(\ov{p-1})-\binom{p}{2}\zeta(1,\ov{p+1})
-(p-1)\zeta(2,\ov{p})-\zeta(3,\ov{p-1}).
\quad\mbox{(for $q=2$)}
\end{align*}

Let $\lambda=-1$ in Eq.\,(\ref{eq.81}). Then 
\[
\sum_{m+n=p}\binom{n}{q}(-1)^{m}\zeta(\ov{m+1})\zeta(\ov{n+1})
=(-1)^{p+q}\zeta(p+1-q,\ov{q+1})
+\sum_{a+b=q}\binom{p-b}{a}\zeta(p+1-b,\ov{b+1}).
\]

The $q=0$ and $q=1$ cases are

\begin{align*}
\sum_{m+n=p}(-1)^m\zeta(\ov{m+1})\zeta(\ov{n+1}) &=
\left[1+(-1)^p\right]\zeta(p+1,\ov{1}),\\
\sum_{m+n=p}(-1)^mn\zeta(\ov{m+1})\zeta(\ov{n+1}) &=
\left[1+(-1)^{p+1}\right]\zeta(p,\ov{2})+p\zeta(p+1,\ov{1}).
\end{align*}

Let $\lambda=1$ in Eq.\,(\ref{eq.81}), we have

\[
\sum_{m+n=p}\binom{n}{q}\zeta(\ov{m+1})\zeta(\ov{n+1})
=\sum_{m+n=p}\binom{n}{q}2^{n-q}\zeta(m+1,\ov{n+1})
+\sum_{\substack{m+n=p\\a+b=q}}\binom{m}{a}\binom{n}{b}
2^{n-b}\zeta(m+1,\ov{n+1}).
\]

We set $q=0$ in the above identity, we will get \cite[Eq.\,(2.14)]{CE2021}

\[
\sum_{m+n=p}\zeta(\ov{m+1})\zeta(\ov{n+1})
=\sum_{m+n=p}2^{n+1}\zeta(m+1,\ov{n+1}).
\]

The well-known double-stuffle relation \cite[Eq.\,(2.2)]{Teo2018}
\begin{equation}\label{eq.astuffle}
\zeta(\ov{r})\zeta(\ov{s})
=\zeta(\ov{r},\ov{s})+\zeta(\ov{s},\ov{r})+\zeta(r+s),
\end{equation}
where $r\geq 1$, $s\geq 1$.
We use the above relation and the sum formula \cite[Eq.\,(2.14)]{Teo2018}
\[
\sum_{m+n=p}\zeta(\ov{m+1},\ov{n+1})
=\zeta(\ov{1})\zeta(p+1)-\zeta(\ov{1},p+1),
\]
We derive the following 
\[
\sum_{m+n=p}2^n\zeta(m+1,\ov{n+1})
=\frac{p+1}2\zeta(p+2)+\zeta(p+1,\ov{1})+\zeta(\ov{p+2}).
\]
This identity was obtained in \cite[Eq.\,(5.3)]{CE2021}. We indicate this formula
in the first paragraph of this section.

Let $\lambda=-2$ in Eq.\,(\ref{eq.81}), we have

\begin{align*}
\sum_{m+n=p}\binom{n}{q}(-1)^m2^{n-q}\zeta(\ov{m+1})\zeta(\ov{n+1})
&=\sum_{m+n=p}\binom{n}{q}(-1)^{m}\zeta(m+1,\ov{n+1})\\
&\qquad+\sum_{\substack{m+n=p\\a+b=q}}\binom{m}{a}\binom{n}{b}
2^{m-a}\zeta(m+1,\ov{n+1}).
\end{align*}

We set $q=0$ in the above identity, we have

\[
\sum_{m+n=p}(-1)^m2^n\zeta(\ov{m+1})\zeta(\ov{n+1})
=\sum_{m+n=p}\Bigl[2^{m}+(-1)^m\Bigr]\zeta(m+1,\ov{n+1}).
\]

Using the double-stuffle relation Eq\,(\ref{eq.astuffle}) we have 

\begin{align*}
&\frac{(-2)^{p+1}-1}3\zeta(p+2)\\
&=\sum_{m+n=p}[(-2)^n+(-2)^m]\zeta(\ov{m+1},\ov{n+1})
-\sum_{m+n=p}(-1)^n[1+(-2)^m]\zeta(m+1,\ov{n+1}).
\end{align*}
\titleformat{\section}
{\sffamily\color{sectitlecolor}\Large\bfseries\filcenter}{}{2em}{#1}%

\end{document}